\documentclass[oneside,english]{amsart}
\usepackage[T1]{fontenc}
\usepackage[latin9]{inputenc}
\usepackage{amsthm}
\usepackage{amssymb}
\usepackage{setspace}
\usepackage[numbers]{natbib}
\makeatletter
\numberwithin{equation}{section}
\numberwithin{figure}{section}
\theoremstyle{plain}
\newtheorem{thm}{\protect\theoremname}
  \theoremstyle{definition}
  \newtheorem{defn}[thm]{\protect\definitionname}
  \theoremstyle{plain}
  \newtheorem{lem}[thm]{\protect\lemmaname}
  \theoremstyle{plain}
  \newtheorem{cor}[thm]{\protect\corollaryname}

\date{}
\usepackage{tensor}

\newcommand{\Hom}{\operatorname{Hom}}
\newcommand{\val}{\operatorname{val}}

\newcommand{\SL}{\operatorname{SL}}
\newcommand{\Gal}{\operatorname{Gal}}
\newcommand{\Fr}{\operatorname{Fr}}

\newcommand{\nr}{\operatorname{nr}}
\newcommand{\kernel}{\operatorname{ker}}
\newcommand{\tors}{\operatorname{torsion}}
\newcommand{\Hh}{\operatorname{H}}

\newcommand{\image}{\operatorname{image}}

\newcommand{\ad}{\operatorname{ad}}

\makeatother

\usepackage{babel}
  \providecommand{\corollaryname}{Corollary}
  \providecommand{\definitionname}{Definition}
  \providecommand{\lemmaname}{Lemma}
\providecommand{\theoremname}{Theorem}

\begin{document}

\title{A Galois side analogue of a theorem of Bernstein}

\author{Manish Mishra}

\address{Mathematics Center, Ruprecht-Karls-Universität Heidelberg\\
D-69120, Heidelberg, Germany}

\email{manish.mishra@gmail.com}
\begin{abstract}
Let $G$ be a connected reductive group defined over a non archimedean
local field $k$. A theorem of Bernstein states that for any compact
subgroup $K$ of $G(k)$, there are, upto unramified twists, only
finitely many $K$-spherical supercuspidal representations of $G(k)$.
We prove an analogous result on the Galois side of the Langlands correspondence. 
\end{abstract}
\maketitle

\section{\label{sec:Intro}Introduction}

Let $G$ be a connected reductive group defined over a non-archimedean
local field $k.$ Let $X_{\nr}(G(k))$ be the group of \textit{unramified
characters} of $G(k)$ (Definition \ref{def:Chars0}). For a smooth
representation $\pi$ of $G(k)$, the various representations $\pi\otimes\chi$,
$\chi\in X_{\nr}(G(k))$ are called the unramified twists of $\pi$.
A Theorem of Bernstein \citep[Theorem 1.4.2.1]{Roc09} states that
\begin{thm}
[Bernstein]\label{thm:intro}For each compact open subgroup $K$
of $G(k)$, the number of isomorphism classes, up to unramified twists,
of irreducible cuspidal representations of $G(k)$ having non-zero
$K$-fixed vectors is finite. 
\end{thm}
On the other hand, local Langlands conjectures predict that ``packets''
of irreducible admissible representations of $G(k)$ should be parametrized
by \textit{Langlands parameters}, which are \textit{admissible} elements
of $\Hh^{1}(W_{k}^{\prime},\hat{G})$ (Definition \ref{def:para}),
where $W_{k}^{\prime}$ is the Weil-Deligne group and $\hat{G}$ is
the complex dual of $G$ (see \citep{Borel1979}). Under this conjectural
correspondence, supercuspidal representations are expected to correspond
to \textit{discrete Langlands parameters} (Definition \ref{def:Fin2}).
By Langlands philosophy, one should expect a result analogous to Theorem
\ref{thm:intro} on the parameter side. 

Let $W_{k}$ be the Weil group of $k$ and $I_{k}$ its inertia subgroup.
Let $\Fr$ be a Frobenius element in $W_{k}$. Let $J$ be an open
subgroup of $I_{k}$ which is normal in $W_{k}$. Call two Langlands
parameters to be equivalent if they are in the same $\Hh^{1}(W_{k}/I_{k},(Z(\hat{G})^{I_{k}})^{\circ})$
orbit (see (\ref{eq:action})), where $Z(\hat{G})$ is the center
of $\hat{G}$. In Theorem \ref{thm:Fin}, we show that upto this equivalence,
there are only finitely many discrete Langlands parameters which are
trivial on $J$. In Section \ref{sec:CharPara}, using Kottwitz homomorphism,
we obeserve that $\Hh^{1}(W_{k}/I_{k},(Z(\hat{G})^{I_{k}})^{\circ})$
is isomorphic to the group $X_{\nr}(G(k))$ of unramified characters
of $G(k)$. 

These statements are thus consistent with the conjectures in \citep[Section 10.3 (2)]{Borel1979}.

\section{Notations}

Let $k$ be a non-archimedean local field and fix an algebraic closure
$\bar{k}$ of $k$. Let $W_{k}$ denote the Weil group of $k$ and
$I_{k}$ denote its inertia subgroup. We fix a Frobenius element $\Fr$
in $W_{k}$. For any algebraic group $\mathcal{G}$, we will denote
by $Z(\mathcal{G})$ the center of $\mathcal{G}$. For any subgroup
$\mathcal{H}$ of $\mathcal{G}$, we will denote by $Z_{\mathcal{G}}\mathcal{H}$,
the centralizer of $\mathcal{H}$ in $\mathcal{G}$. The identity
component of $\mathcal{G}$ will be denoted by $\mathcal{G}^{\circ}$.

\section{\label{sec:WeilRep}Representations of the Weil group}

Let $J$ be an open subgroup of $I_{k}$ which is normal in $W_{k}$. 
\begin{defn}
A representation of $W_{k}$ is called unramified, if it is trivial
on $I_{k}$. \end{defn}
\begin{lem}
\label{lem:WeilRep}Upto unramified twists, there exist only finitely
many irreducible representations of $W_{k}$ which are trivial on
$J$. \end{lem}
\begin{proof}
Let $(\rho,V)$ be an irreducible representation of $W_{k}$ such
that $J\subset$ $\ker\rho$. Let Fr $\in W_{k}$ be a Frobenius element
in $W_{k}$. It acts by conjugation on the finite group $I_{k}/J$,
so some power $\Fr{}^{d}$, $d\geq1$ acts trivially. Thus $\rho(\Fr^{d})$
commutes with $\rho(I_{k})$ and since $\rho$ is irreducible, $\rho(\Fr^{d})$
must be scalar by Schur's lemma. Let $\chi$ be an unramified character
of $W_{k}$ such that $\chi(\Fr^{d})=\rho(\Fr^{d})$. Thus $\rho$
is of the form $\chi\otimes\tau$ where $\tau$ is an irreducible
representation of the finite group $W_{k}/\langle\Fr^{d},J\rangle$.
Thus there are upto unramified twists, only finitely many irreducible
representations of $W_{k}$ which are trivial on $J$. 
\end{proof}

\section{\label{sec:para}Langlands parameters}

Let $G$ be a connected, reductive group over $k$ and let $k_{0}$
be the splitting field in $\bar{k}$ of the quasi-split inner form
of $G$. Let $\tensor[^{L}]G{}=\hat{G}\rtimes\Gal(k_{0}/k)$, where
$\hat{G}$ is the complex dual of $G$. The center $\tensor[^{L}]Z{}$
of $\tensor[^{L}]G{}$ is the group of $\Gal(k_{0}/k)$-fixed points
in the center of $\hat{G}$. 
\begin{defn}
\label{def:para}A homomorphism $\varphi:W_{k}\times\SL(2,\mathbb{C})\rightarrow\tensor[^{L}]G{}$
is called \textit{admissible} if 
\begin{enumerate}
\item $\varphi:\SL(2,\mathbb{C})\rightarrow\hat{G}$ is a homomorphism of
algebraic groups over $\mathbb{C}$. 
\item $\varphi$ is continuous on $I_{k}$ and $\varphi(\Fr)$ is semisimple. 
\item The composite $W_{k}\rightarrow^{L}G\rightarrow\Gal(k_{0}/k)$ is
the canonical surjection $W_{k}\rightarrow\Gal(k_{0}/k)$. 
\end{enumerate}

Two admissible homomorphisms are equivalent if they are conjugate
by $\hat{G}$. A \textit{Langlands parameter} is an equivalence class
of admissible homomorphisms. 

\end{defn}
The group $W_{k}^{\prime}:=W_{k}\times\SL(2,\mathbb{C})$ is called
the Weil-Deligne group of $k$. It is sometimes more convenient to
see a Langlands parameter as an element of $\Hh^{1}(W_{k}^{\prime},\hat{G})$.

\begin{defn}
\label{def:para2}A Langlands parameter is \textit{unramified} if
it is trivial on $I_{k}$ and $\SL(2,\mathbb{C})$. 
\end{defn}

\section{\label{sec:Finite}Finiteness result}

Let the notations be as in Section \ref{sec:para}. So $G$ is as
before a connected reductive group defined over $k$. Let $\hat{\mathfrak{g}}$
be the Lie algebra of the complex dual $\hat{G}$ of $G$. Let $W_{k}^{\prime}:=W_{k}\times\SL(2,\mathbb{C})$
denote the Weil-Deligne group of $k$. Let $J$ be an open subgroup
of $I_{k}$ which is normal in $W_{k}$. Let $\Phi(G)$ denote the
set of Langlands parameters of $G$. 

We have a well defined action 
\begin{equation}
\Hh^{1}(W_{k},Z(\hat{G}))\times\Phi(G)\rightarrow\Phi(G),\qquad[\alpha]\cdot[\phi]\mapsto[\alpha\cdot\phi]\label{eq:action}
\end{equation}

\begin{defn}
\label{def:Fin1}Call two parameters $\varphi$, $\varphi^{\prime}$
to be equivalent if they are in the same $\Hh^{1}(W_{k}/I_{k},(Z(\hat{G})^{I})^{\circ})$
orbit. \end{defn}
\begin{lem}
\label{lem:Fin2}Let $T$ be a tori defined over $k$ and let $\hat{T}$
be its comples dual. There are only finitely many equivalence classes
of Langlands parameters for $T$ which are trivial on $J$. \end{lem}
\begin{proof}
We have a canonical decomposition $\Hh^{1}((\langle\Fr\rangle\ltimes I_{k})/J,\hat{T})=\Hh^{1}(\langle\Fr\rangle,(\hat{T}^{I_{k}})^{\circ})\times\Hh(I_{k}/J,\hat{T})$.
Let $d_{J}=|I_{k}/J|$. Then any element of $\Hh^{1}(I_{k}/J,\hat{T})$
is killed by $d_{J}$. Thus the image of these elements lies in the
$d_{J}$-torsion points of $\hat{T}$ which is a finite set. Therefore
$\Hh^{1}(I_{k}/J,\hat{T})$ is finite. \end{proof}
\begin{lem}
\label{lem:Fin3}Let $\varphi:W_{k}^{\prime}\rightarrow\tensor[^{L}]G{}$
be an admissible homomorphism which is trivial on $J$. If $\image(\varphi)$
is not contained in any proper parabolic subgroup of $\tensor[^{L}]G{}$,
then there exists a number $n=n(J,G)$ such that $\varphi(\Fr^{n})\in Z(\tensor[^{L}]G{})$. \end{lem}
\begin{proof}
Let $d$ be a positive integer such that $\Fr^{d}$ acts trivially
on $I_{k}/J$. Then $\varphi(\Fr^{d})\in Z(\image(\varphi))$. Let
$l=|\Gal(k_{0}/k)|$. Then $s:=\varphi(\Fr^{d})^{l}\in\hat{G}$. Let
$H=Z_{\tensor[^{L}]G{}}(s)$. Then $\image(\varphi)\subset H$ and
$s\in Z(H)$. The group $Z_{\tensor[^{L}]G{}}(Z(H)^{\circ})$ is a
Levi subgroup of $\tensor[^{L}]G{}$ containing $H$ and therefore
must be $\tensor[^{L}]G{}$ since $\image(\varphi)$ is not contained
in any proper parabolic subgroup. Thus $Z(H)^{\circ}\subset Z(\tensor[^{L}]G{})$.
From the structure theorem of the centralizers of semisimple elements,
we know that there can be only finitely many possibilities for $H$
\citep[Prop. 2.1]{Kurtzke}. Since $s\in Z(H),$ the fact that there
are only finitely many possibilities for $H$ allows us to choose
a positive integer $a=a(G)$ independent of $H$ such that $s^{a}\in Z(H)^{\circ}$.
The Lemma follows. \end{proof}
\begin{defn}
\label{def:Fin2}A Langlands parameter is called discrete if its image
is not contained in any parabolic subgroup of $\tensor[^{L}]G{}$. \end{defn}
\begin{thm}
\label{thm:Fin}Let $G$ be a connected reductive group over $k$.
Then there exist only finitely many equivalence classes of discrete
Langlands parameters for $G$ which are trivial on $J$. \end{thm}
\begin{proof}
Let $\varphi:W_{k}^{\prime}\rightarrow\tensor[^{L}]G{}$ be a an admissible
homomorphism. By Lemma \ref{lem:Fin3}, there exists an integer $n=n(J,G)$
such that the composite map $\bar{\varphi}:W_{k}^{\prime}\rightarrow\tensor[^{L}]G{}\rightarrow\mathcal{G}:=\hat{G}_{\ad}\rtimes\Gal(k_{0}/k)$
factors through $W_{k}/\langle\Fr^{n},J\rangle\times\SL(2,\mathbb{C})$.
By \citep[II.3, Theorem 1]{Slodowy}, there are only finitely many
$\mathcal{G}$ conjugacy classes of homomorphisms $W_{k}/\langle\Fr^{n},J\rangle\rightarrow\mathcal{G}$.
It follows that there are only finitely many $\mathcal{G}$ conjugacy
classes of homomorphisms $W_{k}^{\prime}\rightarrow\mathcal{G}$ which
are trivial on $J$. 

Now if $\varphi_{1},\varphi_{2}\in\Hh^{1}(W_{k}^{\prime},\hat{G})$
are two Langlands parameters such that their images in $\Hh^{1}(W_{k}^{\prime},\hat{G}_{\ad})$
are equal, then $\varphi_{1}=\varphi_{c}\cdot\varphi_{2}$ where $\varphi_{c}\in\Hh^{1}(W_{k},Z(\hat{G}))$.
By Lemma \ref{lem:Fin2}, there are only finitely many such $\varphi_{c}$
upto equivalence. The theorem follows.
\end{proof}
The Weil group $W_{k}$ carries an upper numbering filteration $\{W_{k}^{r}\}_{r\geq0}$.
The depth of a parameter $\varphi:W_{k}^{\prime}\rightarrow\tensor[^{L}]G{}$
is defined to be 
\[
\inf\{r\geq0:W_{k}^{s}\subset\ker(\varphi)\mbox{ for }s>r\}.
\]

\begin{cor}
There exist only finitely many equivalence classes of Langlands parameters
of a given depth. 
\end{cor}

\section{\label{sec:chars}Unramified characters}

Let $X_{k}(G)=\Hom(G,\mathbb{G}_{m})$, the lattice of $k$-rational
characters of $G$. Let 
\[
G(k)^{1}=\{g\in G(k):\val_{k}(\chi(g))=0,\forall\chi\in X_{k}(G)\}.
\]
 Then $G(k)^{1}$ is an open normal subgroup of $G(k)$ that contains
each compact subgroup of $G(k)$. It also has the following properties:
\begin{enumerate}
\item $G(k)^{1}$ has compact center;
\item $G(k)/G(k)^{1}$ is a free abelian group of finite rank;
\item The center $Z(G(k)^{1})$ of $G(k)^{1}$ has finite rank in $G(k)$. \end{enumerate}
\begin{defn}
\label{def:Chars0}The group $X_{\nr}(G(k))$ of \textit{unramified
characters} of $G(k)$ is defined by 
\[
X_{\nr}(G(k))=\Hom(G(k)/G(k)^{1},\mathbb{C}^{\times}).
\]

\end{defn}

\begin{defn}
\label{def:Chars}For a smooth representation $\pi$ of $G(k)$, the
representations $\pi\otimes\chi,$ $\chi\in X_{\nr}(G(k))$ are called
the \textit{unramified twists} of $\pi$. \end{defn}
\begin{thm}
[Bernstein]For each compact open subgroup $K$ of $G(k)$, the number
of isomorphism classes, up to unramified twists, of irreducible cuspidal
representations $\tau$ of $G(k)$ with $\tau^{K}\neq0$ is finite. 
\end{thm}

\section{\label{sec:CharPara}Langlands parameters for unramified characters}

In \citep[Section 7]{Kott97}, Kottwitz defined a surjective homomorphism
\[
\kappa_{G}:G(k)\rightarrow X^{*}(Z(\hat{G}))_{I_{k}}^{\Fr}.
\]
 Let $v_{G}:G(k)\rightarrow X^{*}(Z(\hat{G}))_{I_{k}}^{\Fr}/\tors$
be the homomorphism induced by the Kottwitz homomorphism. Then $\kernel v_{G}=G(k)^{1}$
(see \citep[Remark 10]{Ha08}). We therefore have:
\begin{eqnarray}
X_{\nr}(G(k)) & \cong & \Hom(X^{*}(Z(\hat{G}))_{I_{k}}^{\Fr}/\tors,\mathbb{C}^{\times})\nonumber \\
 & \cong & \Hom(X^{*}((Z(\hat{G})^{I_{k}})^{\circ}{}_{\Fr}),\mathbb{C}^{\times})\nonumber \\
 & \cong & (Z(\hat{G})^{I_{k}})^{\circ}{}_{\Fr}.\label{eq:CharPara2}
\end{eqnarray}
 The last equality holds by Cartier duality. 

We have 
\begin{equation}
(Z(\hat{G})^{I_{k}})_{\Fr}^{\circ}\cong\Hh^{1}(W_{k}/I_{k},(Z(\hat{G})^{I_{k}})^{\circ})\hookrightarrow\Hh^{1}(W_{k},Z(\hat{G})).\label{eq:CharPara3}
\end{equation}

Combining equations (\ref{eq:CharPara2}) and (\ref{eq:CharPara3}),
we get a map
\begin{equation}
X_{\nr}(G(k))\hookrightarrow\Hh^{1}(W_{k},Z(\hat{G}))\rightarrow\Hh^{1}(W_{k}^{\prime},\hat{G}).
\end{equation}

One can thus associate to the unramified characters, Langlands parameters
whose image lie in the center of $\hat{G}$.

\section*{Acknoledgement}

The author is very thankful to Sandeep Varma for pointing out several
mistakes in the earlier drafts of this article. He is also very thankful
to Dipendra Prasad for pointing out a gap in the proof of Theorem
\ref{thm:Fin} and for his careful proof reading. He is grateful to
his host Rainer Weissauer for many helpful suggestions and to the
Mathematics Center Heidelberg (MATCH) where this work was written.
He would also like to thank Sudhanshu Shekhar for many helpful conversations.

\section*{Bibliography}

\bibliographystyle{alpha}
\bibliography{fin}

\end{document}